%% file: top.tex
\documentclass[10pt,twoside,reqno,a4paper]{amsart}
\pdfoutput=1
\usepackage[utf8]{inputenc}
\usepackage{courier}
\usepackage[T1]{fontenc}

\usepackage[mathbf]{euler}
\usepackage{topformat}
\usepackage[driver=pdftex,margin=3cm,heightrounded=true,centering]{geometry}
\usepackage{hyperref}
\usepackage{cleveref}
\usepackage{enumerate}
\usepackage{transparent}
\usepackage{amssymb}
\usepackage{amsopn}
\usepackage{amsmath}
\usepackage{mathtools}
\usepackage[backend=biber, style=alphabetic, url=false, maxbibnames=99,
sorting=nyt]{biblatex}
\usepackage{tikz}
\usepackage{tikz-cd}
\usepackage{ifthen}
\usepackage{graphics}
\usepackage{amsthm}
\usepackage{topthm}
\author{Thorben Kastenholz}
\thanks{This research is funded by the Deutsche Forschungsgemeinschaft (DFG,
German Research Foundation) – Project number 281869850}
\date{\today}
\title{Bounded cohomology, Codimension two submanifolds and Pontryagin-Thom
constructions}
\address{Karlsruher Institut f\"ur Technologie, Englerstraße 2, 76131
  Karlsruhe, Germany}
\email{tkastenholz@gmx.de}

\begin{document}
\input{commands.tex}

\input{Abstract.tex}
\maketitle
\section{Introduction}
\input{Section/Introduction.tex}
\section{Splitting manifolds and bounded cohomology}
\input{Section/strategy.tex}
\section{Codimension 2 subspaces and their regular neighborhoods}
\input{Section/Codim.tex}
\section{Complements, Fundamental Groups and Pontryagin-Thom Constructions}
\input{Section/Construction.tex}
\printbibliography
\end{document}

%% file: commands.tex
\newcommand{\introduce}[1]
  {\textbf{#1}}
\newcommand{\tk}[1]{\todo[size=\tiny,color=green!40]{TK: #1}}
\newcommand\blfootnote[1]{%
  \begingroup
  \renewcommand\thefootnote{}\footnote{#1}%
  \addtocounter{footnote}{-1}%
  \endgroup
}

\newcommand{\apply}[2]
  {{#1}\!\left({#2}\right)}
\newcommand{\at}[2]
  {\left.{#1}\right\rvert_{#2}}
\newcommand{\Identity}%
  {\mathrm{Id}}
\newcommand{\NaturalNumbers}%
  {\mathbf{N}}
\newcommand{\Integers}%
  {\mathbf{Z}}
\newcommand{\Rationals}%
  {\mathbf{Q}}
\newcommand{\Reals}%
  {\mathbf{R}}
  \newcommand{\ComplexNumbers}%
  {\mathbf{C}}
\newcommand{\AbstractProjection}[1] 
  {p_{#1}}
\newcommand{\RealPart}[1]
  {\apply{\operatorname{Re}}{#1}}
\newcommand{\ImaginaryPart}[1]
  {\apply{\operatorname{Im}}{#1}}
\newcommand{\Floor}[1]
  {\left \lfloor #1 \right \rfloor}
\newcommand{\Norm}[1]
  {\left|\left|#1\right|\right|}
\newcommand{\MappingCone}[2]
  {\apply{\mathrm{Cone}}{#1,#2}}

\newcommand{\Surface}[1]
  {\Sigma_{#1}}
\newcommand{\SurfaceGroup}[1]
  {S_{#1}}
\newcommand{\Manifold}%
  {M}
\newcommand{\NonOrientableManifold}
  {N}
\newcommand{\Tangenbundle}[1]
  {T#1}
\newcommand{\NormalBundle}[2]
  {N_{#2}#1}
\newcommand{\FiberTransferHomology}[1]
  {#1^{!}}
\newcommand{\Submanifold}%
  {N}
\newcommand{\TubularNeighborhood}[1]
  {U_{#1}}
\newcommand{\ManifoldAlternative}%
  {P}
 \newcommand{\ManifoldAuxiliary}%
  {K}
\newcommand{\NullBordism}%
  {W}
\newcommand{\Bordism}
  {P}
 \newcommand{\ManifoldFiber}%
 {M}
 \newcommand{\ManifoldTotal}%
 {E}
  \newcommand{\ManifoldBase}%
  {B}
\newcommand{\SmoothMap}%
  {\phi}
  \newcommand{\MorseFunction}%
  {f}
\newcommand{\Diffeomorphism}%
  {\Phi}
\newcommand{\Dimension}%
  {d}
  \newcommand{\HalfDimension}%
  {n}
\newcommand{\FundamentalClass}[1]
  {\left[#1\right]}
\newcommand{\Interval}%
  {I}
\newcommand{\Ball}[1]
  {D^{#1}}
\newcommand{\Sphere}[1] 
  {S^{#1}}
\newcommand{\Torus}[1]
  {T^{#1}}
\newcommand{\SimplicialVolume}[1]
  {\lvert \lvert #1 \rvert \rvert}
\newcommand{\ellone}%
  {\ell_{1}}
\newcommand{\Boundary}[1]
  {\partial #1}
\newcommand{\ComplexOfEmbeddings}[1]
  {\apply{K}{#1}}
\newcommand{\GenusOf}[1]
  {\apply{\Genus}{#1}}
\newcommand{\StableGenusOf}[1]
  {\apply{\overline{\Genus}}{#1}}
\newcommand{\Surgery}%
  {\natural}
\newcommand{\HandleEmbedding}%
  {\Phi}
\newcommand{\Singularity}%
  {X}
\newcommand{\Grassmannian}[3]
  {\apply{\mathrm{Gr}_{#1}}{#2^{#3}}}
\newcommand{\TautologicalBundle}[1]
  {\gamma_{#1}}
\newcommand{\ThomSpace}[1]
  {\apply{\mathrm{Th}}{#1}}
\newcommand{\MappingCylinder}[1]
  {\apply{\mathrm{Cyl}}{#1}}
\newcommand{\ClassifyingSpaceWithMapping}[1]
  {X_{#1}}
\newcommand{\ClassifyingSpaceWithMappingHigherStratum}[2]
  {X_{#1}^{#2}}
\newcommand{\MappingSpace}[2]
  {\apply{\mathrm{Map}}{#1,#2}}
\newcommand{\Evaluation}
  {\mathrm{ev}}

\newcommand{\DiskBundle}[1]
  {\apply{\mathrm{D}}{#1}}
\newcommand{\SphereBundle}[1]
  {\apply{\mathrm{S}}{#1}}

\newcommand{\Diff}[1]
  {\mathrm{Diff}\!\left(#1\right)}
\newcommand{\DiffGroup}[1]
  {\mathrm{Diff}^{B\Group}\!\left(#1\right)}
\newcommand{\DiffZero}[1]
  {\mathrm{Diff}_0\!\left(#1\right)}
\newcommand{\DiffOne}[1]
  {\widetilde{\mathrm{Diff}}_0\!\left(#1\right)}
\newcommand{\HomeoGroup}[1]
  {\apply{\mathrm{Homeo}}{#1}}
\newcommand{\HomeoCompactGroup}[1]
  {\apply{\mathrm{Homeo}_{c}}{#1}}
\newcommand{\HomeoLowerGroup}[1]
  {\apply{\mathrm{Homeo}^{\geq}}{#1}}

\newcommand{\FiberingSpace}%
  {E}
\newcommand{\FiberingProjektion}[1] 
  {\pi_{#1}}
\newcommand{\Fiber}%
  {F}
\newcommand{\FiberDimension}%
  {d}
\newcommand{\Base}%
  {B}
\newcommand{\ClutchingFunction}[1] 
  {\varphi_{#1}}

\newcommand{\Group}%
  {\Gamma}
\newcommand{\Subgroup}
  {H}
\newcommand{\AmenableGroup}
  {A}
\newcommand{\GroupElement}%
  {g}
\newcommand{\Genus}%
  {g}
\newcommand{\QuadraticModule}%
  {\mathbf{M}}
\newcommand{\WittIndex}[1]
  {\apply{\Genus}{#1}}
\newcommand{\StableWittIndex}[1]
  {\apply{\overline{\Genus}}{#1}}
\newcommand{\ComplexOfHyperbolicInclusions}[1]
  {\apply{K^{a}}{#1}}
\newcommand{\ChainContraction}[1]
  {H_{#1}}
\newcommand{\FreeGroup}[1]
  {F_{#1}}
\newcommand{\BraidGroup}[2]
  {\apply{B_{#1}}{#2}}
\newcommand{\PureBraidGroup}[2]
  {P\BraidGroup{#1}{#2}}
\newcommand{\Multidiagonal}[2]
  {\apply{\Delta_{#1}}{#2}}

\newcommand{\HomologyClass}%
  {\alpha}
\newcommand{\HomologyOfSpaceObject}[3]
  {\apply{H_{#1}}{#2 ; #3}}
\newcommand{\CohomologyOfSpaceObject}[3]
  {\apply{H^{#1}}{#2 ; #3}}
\newcommand{\BoundedCohomologyOfSpaceObject}[3]
  {\apply{H^{#1}_{\text{b}}}{#2 ; #3}}
\newcommand{\BoundedCohomologyOfSimplicialObject}[3]
  {\apply{H^{#1}_{\text{b, s}}}{#2 ; #3}}
\newcommand{\HomologyOfSpaceMorphism}[1]
  {{#1}_{\ast}}
\newcommand{\HomologyOfGroupObject}[3]
  {\apply{H_{#1}}{#2; #3}}
\newcommand{\HomologyOfGroupMorphism}[2]
  {{#1}_{\ast}}
\newcommand{\HomologyOfSpacePairObject}[3]
  {\apply{H_{#1}}{{#2},{#3}}}
\newcommand{\Multiple}%
  {\lambda}

\newcommand{\TopologicalSpace}%
  {X}
\newcommand{\MappingTorus}[1]
  {T_{#1}}
\newcommand{\Inner}[1]
  {\mathring{#1}}
\newcommand{\Point}%
  {\ast}
\newcommand{\Loop}%
  {\gamma}
\newcommand{\ContinuousMap}%
  {f}
  \newcommand{\ContinuousMapALT}%
  {g}
 \newcommand{\maps}%
  {\ensuremath{\text{maps}}}
\newcommand{\HomotopyGroupOfObject}[3]
  {\apply{\pi_{#1}}{{#2},{#3}}}
\newcommand{\HomotopyGroupOfPairObject}[4]
  {\apply{\pi_{#1}}{{#2},{#3},{#4}}}
\newcommand{\HomotopyGroupMorphism}[1] 
  {{#1}_{\ast}}
\newcommand{\EMSpace}[2]
  {\apply{K}{{#1},{#2}}}
\newcommand{\ClassifyingSpace}[1] 
  {B#1}
\newcommand{\UniversalCovering}[1] 
  {\widetilde{#1}}
\newcommand{\UniversalCoveringMap}[1] 
  {\widetilde{#1}}
\newcommand{\FundamentalCycle}[1]
  {\sigma_{#1}}

\newcommand{\SimplicialComplex}
  {X}
\newcommand{\AuxSimplicialComplex}
  {K}
\newcommand{\Subcomplex}
  {Y}
\newcommand{\AuxSubcomplex}
  {L}
\newcommand{\Simplex}[1]
  {\sigma_{#1}}
\newcommand{\Link}[2]
  {\apply{\text{Lk}_{#1}}{#2}}
\newcommand{\Star}[2]
  {\apply{\text{St}_{#1}}{#2}}
\newcommand{\BoundaryIndexSimplex}[2]
  {\apply{\partial_{#1}}{#2}}
\newcommand{\BoundarySimplex}%
  {\partial}
\newcommand{\StandardSimplex}[1]
  {\Delta_{#1}}
\newcommand{\vertex}
  {v}
\newcommand{\GeometricRealization}[1]
  {\left\lvert #1 \right\rvert}
\newcommand{\BoundHomotopy}
  {N}
\newcommand{\Horn}[2]
  {\Lambda^{#1}_{#2}}

%% file: Abstract.tex
\begin{abstract}
  In this note we develop a novel approach for proving the non-vanishing of
  bounded cohomology. This utilizes a splitting argument whose simplest form is
  as follows: Let $\Manifold$ denote an $n$-manifold of non-zero simplicial
  volume and $\Submanifold$ a codimension two submanifold of $\Manifold$, then
  one can
  conclude that the $n$-th bounded cohomology of the fundamental group of
  $\Manifold \setminus \Submanifold$ is non-zero. We then translate the
  existence of a complement with a given fundamental group into
  an easily accessible homology computation, which might be of independent
  interest.
\end{abstract}

%% file: Section/Introduction.tex
Ever since Gromovs seminal paper \cite{GromovBoundedCohomology}, bounded
cohomology has been known to be a powerful albeit hard to compute invariant
that has many different applications in geometry, topology, group theory and
dynamics. For computations, the key difference between ordinary cohomology and
bounded cohomology is that bounded cohomology does neither satisfy excision nor
Mayer-Vietoris. This is most famously encapsulated in the mysterious bounded
cohomology of free groups: the bounded cohomology of a single circle
vanishes, while the bounded cohomology of the wedge of two circles represents
the bounded cohomology of the free group $\FreeGroup{2}$. It is known that the
bounded cohomology of free groups does not vanish in degrees $2$ (see
\cite{BrooksSecondBC}) and degree $3$ (\cite{SomaThird}). So far nothing is
known about higher degrees.

In this paper we develop a novel approach to
establish non-vanishing results for bounded cohomology based on manifolds with
non-zero simplicial volume.
Soma's proof of the non-vanishing of the third bounded cohomology of free
groups in \cite{SomaThird} has used the uniform boundary condition introduced
by Matsumoto and Morita in \cite{MatsumotoMoritaBoundedCohomology}. Our
proposed strategy is founded on this definition as well. Paraphrased our
approach is based on the following Mayer-Vietoris principle (See
Proposition~\ref{prp:Splitting}) for a
precise statement): If an $n$-manifold of non-zero simplicial volume can
be split into two parts, one of which has vanishing simplicial volume, then the
other one has to have non-vanishing $n$-th bounded cohomology.

\paragraph{Codimension 2 subspaces}
In Section~\ref{scn:Codim} we will prove that the simplicial volume of regular
neighborhoods of nicely immersed codimension $2$-subspaces vanishes
(see Proposition~\ref{prp:TubNeighborhood} for the precise statement). Combined
with the aforementioned strategy, we obtain as a corollary the following theorem
(This represents a weakened version, see Theorem~\ref{thm:ImmersedKnotsEtc}
for a more general statement):
\input{Theorem/KnotsIntroduction.tex}
As a first application let us reprove that the third bounded cohomology of free
groups is non-zero. Indeed it is a well-known fact that every 3-manifold admits
an open book decomposition, hence there exists a link in every three-manifold
such that the complement fibers over a circle (in fact, since we only need a
single example, a hyperbolic dehn surgery on a fibering hyperbolic knot like
the figure eight knot suffices for what follows). Therefore the fundamental
group
$G$ of the complement, which has non-zero third bounded cohomology by
\autoref{thm:KnotGroup}, fits into a short exact sequence of the form
\[
  0
  \to
  \FreeGroup{k}
  \to
  G
  \to
  \Integers
  \to
  0
\]
From this it follows easily that the third bounded cohomology of free groups is
non-zero. Note that \cite{KastenholzOpenBooks} establishes that the very same
approach cannot work in degree $4$.

\paragraph{Translation to algebraic topology}
Finding submanifolds and computing the fundamental groups of their complements
is complicated. In \Cref{scn:Construction} we translate this into a lifting
problem, which is easily computable and might be of independent interest for
algebraic and geometric topologists. This has the following
shape: Fix a group
$\Group$, a rational homology class $\HomologyClass$ of $\Group$ and a
surjection $s\colon H \to \Group$. Then, there exists a certain Thom space
naturally associated to the surjection
$\ThomSpace{\TautologicalBundle{2,\Group,H}}$
together with a map
\[
  \ThomSpace{\TautologicalBundle{2,\Group,H}}
  \to
  \MappingCone{\ClassifyingSpace{\Group}}{\ClassifyingSpace{H}}
\]
where $\MappingCone{\ClassifyingSpace{\Group}}{\ClassifyingSpace{H}}$ denotes
the mapping cone of the map induced by $s$. Now $\HomologyClass$ admits a
representing manifold $\Manifold$ with fundamental group $\Group$ together with
a submanifold $\Submanifold$ such that the fundamental group of the complement
is isomorphic to $H$ if and only if the image of $\HomologyClass$ in
$\MappingCone{\ClassifyingSpace{\Group}}{\ClassifyingSpace{H}}$ lifts
rationally to $\ThomSpace{\TautologicalBundle{2,\Group,H}}$.

This allows us to prove the following, which might be of independent interest:
\begin{theorem}
  \label{thm:CentralizerDimension}
  Let $\Manifold$ denote a rationally essential manifold of dimension $n$, i.e.
  a manifold such that the canonical map
  $
  \Manifold
  \to
  \ClassifyingSpace{\HomotopyGroupOfObject{1}{\Manifold}{\ast}}
  $
  maps the rational fundamental class of $\Manifold$ to a rational non-zero
  group homology class, and $\Submanifold$ a codimension $2$ submanifold. Let
  $H$ denote the fundamental group of $\Manifold\setminus \Submanifold$ and let
  $s$ denote the induced surjection between $H$ and the fundamental group of
  $\Manifold$. Assume that both groups are torsion-free, then there exists at
  least one non-zero element in the kernel of
  $s$ with a centralizer that has rational cohomological dimension at least
  $n-1$.
\end{theorem}
As an immediate corollary we obtain:
\begin{corollary}
  Let $\Manifold$ denote a rationally essential manifold of dimension $n$ and
  suppose that there exists a codimension two-submanifold $\Submanifold$ in
  $\Manifold$ such that the complement has fundamental group $H$,
  then the rational cohomological dimension of $H$ is at least $n-1$.

  Additionally, if $n$ is at least three, then $H$ is not hyperbolic. In
  particular it can not be a free group.
\end{corollary}

Finally at the end of \Cref{scn:Construction}, we sketch how to incorporate
more general singularities, as allowed in the most general version of our
approach, into this algebraic topological framework. We want to emphasize here
that the
computations for the singular case are tedious, but not complicated, and have
not yet been carried out. All of this provides a potential straight forward
route to prove the non-vanishing of the bounded cohomology of free groups.
Additionally, since the bounded cohomology of many groups embeds into the the
bounded cohomology of free groups (e.g. surface groups, many free
products and as established recently in \cite{3ManifoldGroups} all 3-manifold
groups )
there is a lot of freedom in choosing the $H$ for this approach.

\paragraph*{A personal note:} I want to thank Zixiang Zhou for spotting a
critical gap in a predecessor of this paper, where I wrongfully claimed that I
could prove the non-vanishing of the fourth bounded cohomology of free groups.
This paper grew out of my attempt to fix this gap and flesh out the overall
strategy. I also want to thank Francesco Fournier-Facio, Marco Moraschini and
Roberto Frigerio for many helpful remarks on this predecessor paper that have
found their way into this paper as well.
\newline
Finally, I want to emphasize that I am still confident that the computations
described at the end of \Cref{scn:Construction} are feasible and the only
reason I did not finish them, is me leaving pure math academia. Hopefully,
someone else will be able to see this approach to its conclusion.

%% file: Theorem/KnotsIntroduction.tex
\begin{theorem}
\label{thm:KnotGroup}
  Let $\Manifold$ denote an $n$-dimensional manifold of non-zero simplicial
  volume and let $\Submanifold \subset \Manifold$ denote a submanifold of
  codimension two, then the $n$-th bounded cohomology of $\Manifold \setminus
  \Submanifold$ is non-zero.
\end{theorem}

%% file: Section/strategy.tex
In this section we establish a connection between simplicial volume, bounded
cohomology and splitting a manifold into two parts.
All cohomology and homology groups will have real coefficients, additionally
every occurring fundamental cycle will be a real fundamental cycle.

The approach is based on the following definition by Matsumoto and Morita.
\input{Definition/UBC.tex}
The uniform boundary condition is closely tied to bounded cohomology via the
following proposition:
\input{Proposition/MatsumotoMorita.tex}
Note that by Gromovs Mapping Theorem (proven in \cite{GromovBoundedCohomology},
see also Theorem~5.9 in \cite{FrigerioBook}) the bounded cohomology of a
connected
topological space agrees with the bounded cohomology of its fundamental group.
Hence we obtain the following corollary:
\input{Corollary/AcyclicImpliesUBC.tex}
In order to establish the aforementioned Mayer-Vietoris type proposition,
we
will be interested in constructing fundamental cycles of manifolds using
the
uniform boundary condition. For this we need the
following lemma:
\input{Lemma/BoundingCycle.tex}
Using this lemma together with the uniform boundary condition yields the
following elementary observation about simplicial volume:
\input{Lemma/UBCimpliesSimplicialVolume.tex}
Combining this lemma with Theorem~\ref{prp:MatsumotoMorita} yields the
following proposition, which encapsulates the
main strategy of this paper.
\input{Proposition/NonZeroBoundedCohomology.tex}

%% file: Definition/UBC.tex
\begin{definition}[Definition~2.1 in
\cite{MatsumotoMoritaBoundedCohomology}]
  Let $\TopologicalSpace$ denote a topological space. We say
  $\TopologicalSpace$ satisfies the $q$-uniform boundary condition ($q$-UBC) if
  there exists a constant $K$ such that for every closed singular $q$-chain
  with real coefficients $\sigma$ on $\TopologicalSpace$ that is a boundary,
  there exists a
  $q+1$-chain $\rho$ on $\TopologicalSpace$ such that $\partial \rho = \sigma$
  and $\Norm{\rho}\leq K \Norm{\sigma}$.
\end{definition}

%% file: Proposition/MatsumotoMorita.tex
\begin{theoremnum}[Theorem~2.8 in \cite{MatsumotoMoritaBoundedCohomology}]
\label{prp:MatsumotoMorita}
  A topological space $\TopologicalSpace$ satisfies the $q$-UBC if
  and only if the comparison map
  \[
    \BoundedCohomologyOfSpaceObject{q+1}{\TopologicalSpace}{\Reals}
    \to
    \CohomologyOfSpaceObject{q+1}{\TopologicalSpace}{\Reals}
  \]
  is injective.
\end{theoremnum}

%% file: Corollary/AcyclicImpliesUBC.tex
\begin{corollary}
\label{cor:AcyclicUBC}
  Let $\TopologicalSpace$ denote a connected topological space such that
  $
    \BoundedCohomologyOfSpaceObject
      {q+1}
      {\HomotopyGroupOfObject{1}{\TopologicalSpace}{\ast}}
      {\Reals}
  $
  vanishes, then $\TopologicalSpace$ satisfies the $q$-UBC.
\end{corollary}

%% file: Lemma/BoundingCycle.tex
\begin{lemma}
\label{lem:BoundingCycle}
  Let $\NullBordism$ denote an oriented $n+1$-dimensional compact manifold with
  boundary $\Manifold$ and let $\FundamentalCycle{\Manifold}$ denote a
  fundamental cycle of $\Manifold$. Then any $n+1$-chain
  $\FundamentalCycle{\NullBordism,\Manifold}$ on $\NullBordism$ that bounds
  $\FundamentalCycle{\Manifold}$ represents the fundamental class of
  $(\NullBordism,\Manifold)$.
\end{lemma}
\begin{proof}
  The boundary morphism of the long exact sequence of the pair
  $(\NullBordism,\Manifold)$ maps an $(n+1)$-chain on the pair
  $(\NullBordism,\Manifold)$ to its boundary in $\Manifold$. Hence the homology
  class represented by $\FundamentalCycle{\NullBordism,\Manifold}$ maps to the
  class represented by $\FundamentalCycle{\Manifold}$ i.e. the fundamental
  class of $\Manifold$. Therefore it represents the fundamental class of
  $(\NullBordism,\Manifold)$.
\end{proof}

%% file: Lemma/UBCimpliesSimplicialVolume.tex
\begin{lemma}
\label{lem:UBCGluing}
  Suppose a closed connected $n$-dimensional manifold $\Manifold$ splits as
  $
    \Manifold
    =
    \Manifold_1
    \cup_{\Submanifold}
    \Manifold_2
  $
  with $\Manifold_1$ and $\Manifold_2$ being two codimension $0$ submanifolds
  with boundary, that intersect in their common boundary
  $\Submanifold$. Suppose further that
  $\SimplicialVolume{\Manifold_2,\Submanifold}$
  vanishes and $\Manifold_1$ satisfies $(n-1)$-UBC, then the
  simplicial volume of $\Manifold$ vanishes.
\end{lemma}
\begin{proof}
  Let $(\FundamentalCycle{\Manifold_2}^{k})_{k\in\NaturalNumbers}$ denote a
  sequence of fundamental cycles of $(\Manifold_2,\Boundary{\Manifold_2})$ such
  that
  $\Norm{\FundamentalCycle{\Manifold_2}^{k}}$ tends to zero as $k$ tends to
  infinity.
  Let $K_1$ denote the UBC-constant in dimension $n-1$ of $\Manifold_1$. Now by
  $(n-1)$-UBC there exist an $n$-chain
  $\FundamentalCycle{\Manifold_1}^{k}$ on
  $\Manifold_1$ that bounds $\Boundary{\FundamentalCycle{\Manifold_2}^k}$
  and such that
  \[
    \Norm{\FundamentalCycle{\Manifold_1}^{k}}
    \leq
    K_1
    \Norm{\Boundary{\FundamentalCycle{\Manifold_2}^{k}}}
    \leq
    (n+1)K_1\Norm{\FundamentalCycle{\Manifold_2}^{k}}
  \]
  By Lemma~\ref{lem:BoundingCycle},
  $\FundamentalCycle{\Manifold_1}^{k}$ represents the fundamental
  class of $(\Manifold_1,\Boundary{\Manifold_1})$.
  By construction, the difference of $\FundamentalCycle{\Manifold_1}^{k}$ and
  $\FundamentalCycle{\Manifold_2}^{k}$ is a cycle and therefore represents the
  fundamental class of $\Manifold$. Since the norm of
  $\FundamentalCycle{\Manifold_1}^{k}$ tends to zero, the norm of these
  differences tends to zero as well.
\end{proof}

%% file: Proposition/NonZeroBoundedCohomology.tex
\begin{proposition}
\label{prp:Splitting}
  Let $\Manifold$ denote a closed manifold of dimension $n$ with non-zero
  simplicial volume.
  Suppose further that there exists a codimension $0$ submanifold with boundary
  $\Manifold'\subset \Manifold$ such that the simplicial volume of
  $(\Manifold',\Boundary{\Manifold'})$ vanishes. Then the $n$-th bounded
  cohomology of
  $\HomotopyGroupOfObject{1}{\Manifold\setminus\Manifold'}{\ast}$ is non-zero.
\end{proposition}

%% file: Section/Codim.tex
\label{scn:Codim}
The goal of the following section is to establish
vanishing results for regular neighborhoods of various codimension $2$
subspaces, since these represent the prime examples of the codimension $0$
manifolds in Proposition~\ref{prp:Splitting}. This will imply
Theorem~\ref{thm:ImmersedKnotsEtc}, the more refined version of
Theorem~\ref{thm:KnotGroup}.
Again all cohomology and homology groups will have real coefficients. Similarly
every fundamental cycle is a real fundamental cycle.

Since regular neighborhoods arise from normal bundles, we will construct
fundamental cycles by realizing the fiber transfer on the chain level and
then
gluing these together along the strata of a subspace.

The fiber transfer is sometimes called the Gysin map or the Umkehr map. There
are various definitions, some more geometric, some more algebraic in
flavor. We
will use the following:
Let $\FiberingProjektion{\ManifoldTotal}\colon\ManifoldTotal \to \ManifoldBase$
denote a fiber bundle of oriented
manifolds with fiber $\ManifoldFiber$ of dimension $\FiberDimension$. Then
$
  \FiberTransferHomology{\FiberingProjektion{\ManifoldTotal}}
  \colon
  \HomologyOfSpaceObject{*}{\ManifoldBase}{\Reals}
  \to
  \HomologyOfSpaceObject{*+\FiberDimension}{\ManifoldTotal}{\Reals}
$
is defined by
\[
  \apply
    {\FiberTransferHomology{\FiberingProjektion{\ManifoldTotal}}}
    {\HomologyClass}
  =
  \apply
    {
      PD_{\ManifoldTotal}
      \circ
      \FiberingProjektion{\ManifoldTotal}^*
      \circ
      PD_{\ManifoldBase}^{-1}
    }
    {\HomologyClass}
\]
where $PD_{\ManifoldTotal}$ and $PD_{\ManifoldBase}$ denote the respective
Poincare duality isomorphisms. This extends verbatim to the case, where
$\ManifoldBase$ or $\ManifoldTotal$ have boundary.
If a homology class $\HomologyClass$ is represented by some map
$\ContinuousMap \colon \Manifold \to \ManifoldBase$ i.e.
$
  \HomologyClass
  =
  \apply
    {\HomologyOfSpaceMorphism{\ContinuousMap}}
    {\FundamentalClass{\Manifold}}
$%
, then one easily checks that
$
  \apply
    {\FiberTransferHomology{\FiberingProjektion{\ManifoldTotal}}}
    {\HomologyClass}
  =
  \apply
    {\HomologyOfSpaceMorphism{\overline{\ContinuousMap}}}
    {\FundamentalClass{\ContinuousMap^* \ManifoldTotal}}
$
where $\overline{\ContinuousMap}$ denotes the induced map from the pullback
$\ContinuousMap^* \ManifoldTotal$ to $\ManifoldTotal$.

It is a classical consequence of Gromovs Mapping Theorem that the simplicial
volume vanishes for total spaces of fiberings with fibers of positive
dimension and with amenable fundamental group.
The following lemma can be understood as an extension of this.
\input{Lemma/FiberTransfer.tex}
With this at hand, let us focus on regular neighborhoods. The following
proposition describes their simplicial volume in arbitrary dimensions.
\input{Proposition/NeighborhoodImmersedCodim2.tex}
\begin{remark}
  In Section~11.2 in \cite{FrigerioMoraschiniMutliComplexes} the notion of
  locally coamenable subcomplexes was introduced (this notion was also
  introduced in \cite{GromovBoundedCohomology}). In particular they show that
  if around every point in a codimension two subcomplex $\Submanifold$ of a PL
  manifold $\Manifold$ there exists a small ball $B$ such that the local
  complement $B \setminus (\Submanifold \cap B)$ has amenable fundamental
  group, then the simplicial volume of a regular neighborhood of $\Submanifold$
  vanishes. Since the fundamental group of the Hopf Link is abelian, this also
  implies Proposition~\ref{prp:GeneralizedTubNeighborhood}, but it also
  includes many mores subcomplexes.
\end{remark}
Due to the nature of 1-dimensional and 0-dimensional manifolds, there are more
general subspaces with regular neighborhoods with vanishing simplicial volume.
In particular, their self-intersections do not have to be self-transverse
everywhere and even very specific self-intersections in codimension
$3$ are possible.
\input{Proposition/NeighborhoodImmersedSurface.tex}
As mentioned before, combining these two results with
Proposition~\ref{prp:Splitting}
yields the following theorem:
\input{Theorem/GeneralizedKnots.tex}

%% file: Lemma/FiberTransfer.tex
\begin{lemma}
\label{lem:FiberTransfer}
  Let $\pi\colon \ManifoldTotal \to \ManifoldBase$ denote a fiber bundle with
  fiber $\ManifoldFiber$ being a
  manifold with a potentially empty
  boundary such that $\ManifoldFiber$ and $\Boundary{\ManifoldFiber}$ are
  connected (or empty) and both have amenable fundamental group, then there
  exists chain maps
  \[
    \FiberTransferHomology{\pi}_{d}
    \colon
    C_d(\ManifoldBase)
    \to
    C_{d+f}(\ManifoldTotal, \Boundary{\ManifoldTotal})
  \]
  where $f$ denotes the dimension of the fiber and $\Boundary{\ManifoldTotal}$
  denotes the restriction of the bundle to the fiberwise boundary, of
  degreewise arbitrarily small norm realizing the fiber transfer
  \[
    \HomologyOfSpaceObject{d}{\ManifoldBase}{\Reals}
    \to
    \HomologyOfSpaceObject{d+f}{\ManifoldTotal,\Boundary{\ManifoldTotal}}{\Reals}
  \]
  on the chain level.
\end{lemma}
\begin{proof}
  Fix a fundamental cycle
  $
    \FundamentalCycle{(\ManifoldFiber, \Boundary{\ManifoldFiber})}
  $
  of arbitrarily small norm and for every point in $\ManifoldBase$ a
  trivialization of the fiber over that point (These do not have to satisfy any
  compatibility conditions, they are truly arbitrary).
  Now define
  $\FiberTransferHomology{\pi}_0$ to be the
  image of $\FundamentalCycle{(\ManifoldFiber,
  \Boundary{\ManifoldFiber})}$
  under the these arbitrary trivializations.

  We now proceed inductively. Suppose we are able to construct maps
  $\FiberTransferHomology{\pi}_{*}$ for $*\leq n-1$ of arbitrary small norm
  representing
  the
  fiber transfer and that
  $\apply{\pi_{*}}{\Simplex{*}}$ is supported
  on $\at{\ManifoldTotal}{\Simplex{*}}$ for any singular simplex $\Simplex{*}$.
  Let $\Simplex{n}$ denote a singular $n$-simplex in $\ManifoldBase$. Note that
  $\Boundary{\Simplex{n}}$ is an $n-1$-sphere. Then
  $
        \apply
          {\FiberTransferHomology{\pi}_{n-1}}
          {\Boundary{\Simplex{n}}}
  $
  represents the fundamental cycle of
  $\at{\ManifoldTotal}{\Boundary{\Simplex{n}}}$, therefore
  $
    \Boundary{
        \apply
          {\FiberTransferHomology{\pi}_{n-1}}
          {\Boundary{\Simplex{n}}}
    }
  $
  is a fundamental cycle of
  $\Boundary{\at{\ManifoldTotal}{\Boundary{\Simplex{n}}}}$.
  Let $K_1$ denote the UBC constant in degree $n+f-2$ of
  $
    \StandardSimplex{n} \times \Boundary{\ManifoldFiber}
  $%
  , which exists because the fundamental group of $\Boundary{\ManifoldFiber}$
  is amenable. Let $\rho_{\Simplex{n}}$ denote a filling of
  $
    \Boundary
    {
      \apply
      {\FiberTransferHomology{\pi}_{n-1}}
      {\Boundary{\Simplex{n}}}
    }
  $
  in $\at{\ManifoldTotal}{\Simplex{n}}\times \Boundary{\ManifoldFiber}$ such
  that
  \[
    \Norm{\rho_{\Simplex{n}}}
    \leq
    K_1
    \Norm{
      \Boundary
        {
          \apply
            {\FiberTransferHomology{\pi}_{n-1}}
            {\Boundary{\Simplex{n}}}
        }
    }
  \]
  Now
  $
    \apply
      {\FiberTransferHomology{\pi}_{n-1}}
      {\Boundary{\Simplex{n}}}
    +
    \rho_{\Simplex{n}}
  $
  represents a fundamental cycle of $\Boundary{(\Simplex{n}\times
  \ManifoldFiber)}$ by Lemma~\ref{lem:BoundingCycle}.
  Let $K_2$ denote the UBC constant of
    $
      \Simplex{n}
      \times
      \ManifoldFiber
    $
    in degree $n+f-1$, then there exists a filling $P_{\Simplex{n}}$ of
    $
      \apply
        {\FiberTransferHomology{\pi}_{n-1}}
        {\Boundary{\Simplex{n}}}
      +
      \rho_{\Simplex{n}}
    $
    such that
    \[
      \Norm{P_{\Simplex{n}}}
      \leq
      K_2
      \Norm
        {
          \apply
            {\FiberTransferHomology{\pi}_{n-1}}
            {\Boundary{\Simplex{n}}}
          +
          \rho_{\Simplex{n}}
        }
      \leq
      K_2
      \left(
        (n+1)\Norm{\FiberTransferHomology{\pi}_{n-1}}
        +
        K_1
        (n+1)n\Norm{\FiberTransferHomology{\pi}_{n-1}}
      \right)
    \]
    Since $\Norm{\FiberTransferHomology{\pi}_{n-1}}$ can be chosen to be
    arbitrarily small, we conclude by Lemma~\ref{lem:BoundingCycle} that this
    yields the desired extensions of arbitrarily small norm.
\end{proof}

%% file: Proposition/NeighborhoodImmersedCodim2.tex
\begin{proposition}
\label{prp:GeneralizedTubNeighborhood}
  Let $\Manifold$ denote a closed $n$-manifold and
  $
    \Submanifold
  $
  a compact subset of $\Manifold$ such that:
  \begin{enumerate}[(i)]
  \item
    $\Submanifold$ is the union of two subsets $\Singularity_{n-2}$ and
    $\Singularity_{n-4}$ such that $\Singularity_{n-2}$ is open in
    $\Submanifold$ and $\Singularity_{n-4}$ is the topological boundary of
    $\Singularity_{n-2}$
  \item
    around every point in $\Singularity_{n-2}$, there exists a
    neighborhood and
    a chart in $\Manifold$ such that $\Submanifold$ is mapped to
    $\Reals^{n-2}
    \subset \Reals^n$.
  \item
    around every point in $\Singularity_{n-4}$, there exists a
    neighborhood and
    a chart in $\Manifold$ such that $\Singularity_{n-4}$ is mapped to
    $\Reals^{n-4}\subset \Reals^{n-2}$ and $\Singularity_{n-2}$ is mapped
    to
    two $n-2$-planes in $\Reals^{n}$ that intersect transversly in
    $\Reals^{n-4}$.
  \end{enumerate}
  then there exists a closed neighborhood $\TubularNeighborhood{\Submanifold}$
  of
  $\Submanifold$ such that
  $\SimplicialVolume{\TubularNeighborhood{\Submanifold}}$ vanishes and
  $
  \Manifold\setminus\TubularNeighborhood{\Submanifold}
  \cong
  \Manifold\setminus \Submanifold
  $%
\end{proposition}
The reader should think of a codimension two immersion with self-transverse
double-points. In this case $\Singularity_{n-2}$ represents the points, where
the immersion is injective and $\Singularity_{n-4}$, represents the double
points, which form a codimension four submanifold.
\begin{proof}
  By definition $\Singularity_{n-4}$ is a codimension $4$ submanifold of
  $\Manifold$.
  Let us denote a complement of a small open tubular neighborhood of
  $\Singularity_{n-4}$ in $\Submanifold$ by $\hat{\Submanifold}_{n-2}$.
  Then
  $\hat{\Submanifold}_{n-2}$ is a codimension $2$-submanifold
  with boundary of $\Manifold$. The disk bundle of its normal bundle is a
  $\Ball{2}$-bundle over $\hat{\Submanifold}_{n-2}$. By
  Lemma~\ref{lem:FiberTransfer} it has vanishing simplicial volume.

  By the tranversality condition, the tangent spaces of the
  two $n-2$-planes of $\Singularity_{n-2}$ of a point
  in $\Singularity_{n-4}$ intersect a sphere of the fiber of the normal
  bundle of $\Singularity_{n-4}$ in a Hopf link.

  At the sphere of this normal bundle of a point of $\Singularity_{n-4}$ we
  are now given fundamental cycles of a tubular neighborhood of the two circles
  constituting the Hopf link and in order to obtain a fundamental cycle of
  a whole regular neighborhood of $\Submanifold$, we have
  to extend this to a fundamental cycle of a tubular neighborhood of
  $\Singularity_{n-4}$.
  The construction of such an extension of arbitrarily small norm is
  completely analogous to the proof of
  Lemma~\ref{lem:FiberTransfer} using that the fundamental group of the
  complement of the Hopf link is abelian i.e. one first extends the given
  fundamental cycles to fundamental cycles of the sphere bundle of the
  normal bundle of $\Singularity_{n-4}$ and then extends these to the
  whole disk
  bundle of the normal bundle of $\Singularity_{n-4}$.
\end{proof}

%% file: Proposition/NeighborhoodImmersedSurface.tex
\begin{proposition}
\label{prp:TubNeighborhood}
  Let $\Manifold$ denote a closed $4$-manifold and $\Submanifold$ a
  compact subset of $\Manifold$ such that:
  \begin{enumerate}[(i)]
  \item
    $\Submanifold$ is the union of three sets: $\Singularity_2$,
    $\Singularity_1$ and $\Singularity_0$, such that $\Singularity_2$ is
    open
    in $\Submanifold$,
    and the topological boundary of $\Singularity_2$ is the disjoint union of
    $\Singularity_0$
    and
    $\Singularity_1$.
  \item
    Around every point in $\Singularity_2$, there exists a neighborhood in
    $\Manifold$ and a chart for this neighborhood such that
    $\Singularity_2$ is
    mapped to $\Reals^2 \subset \Reals^4$.
  \item
  \label{itm:Complex}
    $\Singularity_0$ is a discrete set of points such that around every
    such
    point $x$, there exists a diffeomorphism from a small ball around $x$
    to
    $\ComplexNumbers^2$ such that $x$ is mapped to the origin and
    $\Submanifold$ is mapped to a union of complex subspaces of
    $\ComplexNumbers^2$.
  \item
    $\Singularity_1$ is a disjoint union of circles and around every point
    in
    such a circle, there exists a neighborhood in $\Manifold$ and a chart
    such
    that $\Singularity_1$ agrees with $\Reals \subset \Reals^4$ and
    $\Submanifold$ gets mapped to a union of half-planes in $\Reals^4$ that
    intersect in $\Singularity_1$.
  \end{enumerate}
  Then there exists a closed neighborhood
  $\TubularNeighborhood{\Submanifold}$
  of
  $\Submanifold$ such that
  $\SimplicialVolume{\TubularNeighborhood{\Submanifold}}$ vanishes and
  $
    \Manifold\setminus\TubularNeighborhood{\Submanifold}
    \cong
    \Manifold\setminus \Submanifold
  $%
  .
\end{proposition}
The picture the reader should have in mind is that $\Submanifold$ represents an
immersed two-dimensional submanifold that has self-transverse multiple points
that correspond to the points in $\Singularity_0$ as well as self-intersections
that form circles that correspond to the points in $\Singularity_1$. In
particular, it is important that
$\Singularity_0$ and $\Singularity_1$ are disjoint.

For the $\Singularity_1$-case, an example is given by the product of a circle
and a figure eight in a product of surfaces. This represents two tori that
intersect in a circle.
An example for the $\Singularity_0$-case is given by the
union of the two fibers and the diagonal in the product of a surface with
itself.

Since the presence of complex numbers in (\ref{itm:Complex}) might seem
arbitrary, let us first expand on this:
The intersection of the $2$-planes corresponding to $\Singularity_2$ in a small
sphere centered at a point in $\Singularity_0$ yields a, so called, great circle
link in $\Sphere{3}$. The reason that we require these to come from a complex
chart (although no overall complex structure is required!) is due to the fact
that during the proof we need the link complement to have vanishing simplicial
volume.
For such a complex great circle link this is true, while there are even some
general great circle links with hyperbolic complement. We refer the reader to
\cite{WalshGreatCircleLinks} for an introduction to great circle links.
\begin{remark}
  While Proposition~\ref{prp:GeneralizedTubNeighborhood} was implied by
  Theorem~11.2.3 in
  \cite{FrigerioMoraschiniMutliComplexes}, we want to stress that
  Proposition~\ref{prp:TubNeighborhood} does not follow from their theorem. The
  local complements around points in $\Singularity_1$ and $\Singularity_0$
  do not have amenable fundamental groups in general.
\end{remark}
\begin{proof}
  The proof will be analogous to the previous proof. Let us denote the
  complement of a small neighborhood of $\Singularity_0 \cup
  \Singularity_1$ by
  $\hat{\Submanifold}_2$, this is a submanifold (with boundary) of
  $\Manifold$
  and it has a two-dimensional normal bundle.
  Using Lemma~\ref{lem:FiberTransfer}, we can construct a
  representative of the fundamental class of a small tubular neighborhood
  of $\hat{\Submanifold}_2$.

  Let us first consider the points in $\Singularity_0$. Note that the
  intersection of $\Singularity_2$ with a sphere centered around a point
  in
  $\Singularity_0$ looks like a collection of fibers of the Hopf
  fibration. In
  particular, the complement of the tubular neighborhood of
  $\Singularity_2$ in
  this sphere has vanishing simplicial volume as an $\Sphere{1}$-bundle.
  Since
  the intersection of this complement and the tubular neighborhood
  consists of
  tori, which have amenable fundamental group, we can extend the boundary
  of
  the fundamental cycle for the tubular neighborhood of $\hat{\Submanifold}_2$
  to the
  3-sphere centered at the point in $\Singularity_0$. Since $\Ball{4}$ is
  simply-connected, we can extend this cycle to a representative of the
  fundamental class of a regular neighborhood of $\Singularity_2 \cup
  \Singularity_0$.

  We will proceed similarly for the circles in $\Singularity_1$. The
  neighborhood of one of these circles is diffeomorphic to $\Sphere{1}
  \times
  \Ball{3}$ and the intersection of $\Singularity_2$ with
  $\Sphere{1} \times \Sphere{2}$ is a covering of $\Sphere{1}$. We are given a
  fundamental
  cycle
  for a neighborhood of said intersection, which we have to extend to a
  fundamental cycle of $\Sphere{1}\times \Ball{3}$. Since this
  intersection is
  a finite covering, there exists a finite covering of $\Sphere{1}$ such
  that
  the pullback is diffeomorphic to a trivial covering. The transfer of the
  fundamental cycle for the tubular neighborhood of $\hat{\Submanifold}_2$
  yields a fundamental cycle of $\bigsqcup \Sphere{1} \times \Ball{2}
  \subset
  \Sphere{1}\times \Sphere{2}$. The complement is given by the product of
  $\Sphere{1}$ and the complement of finitely many points in $\Sphere{2}$.
  In
  particular this has vanishing simplicial volume and the intersection is
  again
  given by tori. Hence we can extend this to a fundamental cycle of
  $\Sphere{1}\times \Sphere{2}$. Since $\Sphere{1} \times \Ball{3}$ has
  amenable fundamental group, we can extend this to a fundamental cycle of
  $\Sphere{1}\times \Ball{3}$. Pushing this down along the aforementioned
  covering yields a fundamental cycle of $(\Singularity_2 \cup
  \Singularity_0) \cup \Singularity_1$ i.e. all of $\Submanifold$.
\end{proof}

%% file: Theorem/GeneralizedKnots.tex
\begin{theorem}
\label{thm:ImmersedKnotsEtc}
  Let $\Manifold$ denote an $n$-dimensional manifold with non-zero simplicial
  volume and let
  $
    \Submanifold
  $
  denote a codimension two subspace satisfying the hypothesis of
  Proposition~\ref{prp:GeneralizedTubNeighborhood} or if $n=4$ the
  hypothesis of Proposition~\ref{prp:TubNeighborhood}, then
  the $n$-th bounded cohomology of $\Manifold \setminus
  \apply{\ContinuousMap}{\Submanifold}$ is non-zero.
\end{theorem}
\begin{proof}
  We can write $\Manifold$ as the union of a regular neighborhood of
  $\Submanifold$ and the complement of this regular neighborhood. By assumption
  this complement is homotopy equivalent to the complement of $\Submanifold$ in
  $\Manifold$. By
  Proposition~\ref{prp:GeneralizedTubNeighborhood} and
  Proposition~\ref{prp:TubNeighborhood}
  respectively, we obtain that these
  regular neighborhoods are codimension-$0$ submanifolds with vanishing
  simplicial volume. Because $\Manifold$ has non-zero simplicial volume,
  Proposition~\ref{prp:Splitting} implies that the complement of these regular
  neighborhoods has non-zero $n$-th bounded cohomology.
\end{proof}

%% file: Section/Construction.tex
\label{scn:Construction}
The goal of this sections is to translate the conditions of
Theorem~\ref{thm:ImmersedKnotsEtc} into an easily verifiable topological
condition.
This will be accomplished by using a modified Pontryagin-Thom construction.
For this purpose, let us first briefly recall the classical Pontryagin-Thom
constructions (See Chapter~18 in \cite{MilnorStasheff}):

Let $\Manifold$ denote an $n$-dimensional manifold and $\Submanifold\subset
\Manifold$ a submanifold of dimension $n-k$. Then $\Submanifold$ admits a
tubular neighborhood $\TubularNeighborhood{\Submanifold}$.
Let
$
  \Grassmannian{k}{\Reals}{\infty}
$
denote the Grassmannian of $k$-planes in $\Reals^{\infty}$ and let
$\TautologicalBundle{k}$ denote the corresponding tautological bundle. There is
a classifying map of the normal bundle
$\Submanifold \to \Grassmannian{k}{\Reals}{\infty}$, which is covered by a map
$\TubularNeighborhood{\Submanifold} \to \TautologicalBundle{k}$. This map
extends to a map $\Manifold \to \ThomSpace{\TautologicalBundle{k}}$, where
$\ThomSpace{\TautologicalBundle{k}}$ denotes the Thom space of
$\TautologicalBundle{k}$ i.e. the quotient of the corresponding disk bundle by
the corresponding sphere bundle.
Now the
the famous Pontryagin-Thom yields an equivalence
between bordism classes of submanifolds of codimension $k$ in $\Manifold$ and
homotopy classes of maps from $\Manifold$ to
$\ThomSpace{\TautologicalBundle{k}}$.
We will modify this construction by incorporating the desired fundamental group
of the complement of the submanifold. This is based on the following
observation:
\input{Lemma/SubmanifoldComplementSubcomplex.tex}
Fix groups $\Group$ and $H$ and a surjection $s\colon H \to \Group$ and let us
consider the following space:
\input{Definition/MappingSpaceThomSpace.tex}
Suppose we have some homology class
$
  \HomologyClass
  \in
  \HomologyOfGroupObject{n}{\Group}{\Rationals}
$ for $n\geq 2$.
Let $\Manifold$ denote some $n$-manifold that admits a map
$\ContinuousMap\colon \Manifold\to \ClassifyingSpace{\Group}$ that maps the
fundamental class of $\Manifold$ to a non-zero multiple of $\HomologyClass$.
Now, assume that $\Manifold$ admits a codimension two submanifold
$\Submanifold$ such that the fundamental group of $\Manifold \setminus
\Submanifold$ is isomorphic to $H$ and the induced map from $H$ to $\Group$ is
given by $s$. Then by Lemma~\ref{lem:Factorizes}, the tubular neighborhood
$
  \left(
  \TubularNeighborhood{\Submanifold},
  \Boundary{\TubularNeighborhood{\Submanifold}}
  \right)
$
maps to the pair
$
  \left(
  \ClassifyingSpace{\Group},
  \ClassifyingSpace{H}
  \right)
$, therefore there is a classifying map from $\Submanifold$ to
$
  \ClassifyingSpaceWithMapping
    {\left(\ClassifyingSpace{\Group},\ClassifyingSpace{H}\right)}
$
which extends to a map $\Manifold \to
\ThomSpace{\TautologicalBundle{2,\Group,H}}$ and it is easy to see that
therefore $\ContinuousMap$ factorizes through $\overline{\Evaluation}$.
This shows that if a manifold representative of $\HomologyClass$ has a
submanifold with complement $H$, then the composition of the representing map
and the quotient map to
$\MappingCone{\ClassifyingSpace{\Group}}{\ClassifyingSpace{H}}$
factorizes through $\overline{\Evaluation}$. We will show that a converse to
this is also true:

\input{Proposition/RepresentabilityComplement.tex}
This proposition translates the geometric setup (at least in the case with no
singularities) of Theorem~\ref{thm:ImmersedKnotsEtc} into an easy to check
algebraic criterion.

In order to showcase how to compute the homology of the Thom spaces involved,
we prove Theorem~\ref{thm:CentralizerDimension}:
\input{Theorem/NoHyperbolicComplement.tex}
\begin{remark}
This proof shows further that in order to represent homology classes of
higher degree one needs large centralizers in $H$.
Unfortunately, the centralizers cannot be too artificial. Indeed suppose that
one has two surjections $s_1\colon H_1 \to \Group$ and $s_2\colon H_2\to
\Group$ such that $s_1$ factorizes through $s_2$, then a lift as in
Proposition~\ref{prp:RepresentabilityComplement} for $s_1$ yields a lift for
$s_2$, since
the corresponding Thom spaces are natural with respect to the surjections.
Therefore, one cannot for example take products with abelian groups in order to
artificially inflate the centralizers in $H$.
\end{remark}

\paragraph{Incorporating singularities}
In this paragraph will sketch how to include the singularities of
Proposition~\ref{prp:GeneralizedTubNeighborhood} into this setup. While the
resulting
computations for free groups are not complicated, they are tedious and beyond
the scope of this paper.

The classifying space in Definition~\ref{def:ClassyfingSpaceWithMapping} was
used in
order to classify the normal bundle of the codimension two submanifolds we
wanted to investigate. Now imagine a self-transverse double intersection, what
does the normal data look like? There are two two-planes that intersect
trivially, they span the four-dimensional normal space of the intersecting
manifold. When looking at the unit 4-ball of this four-dimensional normal
space, the two two-planes form coordinate disks bounding a Hopf link. If the
complement would have fundamental group $H$, while the surrounding manifold
would have fundamental group $\Group$, then the inner of this 4-ball as well as
a small neighborhood of said Hopf link would map
to $\ClassifyingSpace{\Group}$, while the complement of the small neighborhood
of the Hopf link in $\Sphere{3}$ would map to $\ClassifyingSpace{H}$. In other
words, when
looking at the contribution of the transverse double point to the boundary of a
regular neighborhood of the singular subspace, one sees the complement of a
Hopf link and they map accordingly to the corresponding classifying spaces of
$\Group$ and $H$ respectively.

Therefore we define $\ClassifyingSpaceWithMappingHigherStratum{\Group,H}{(DP)}$
as the classifying space for tuples consisting of a $4$-dimensional bundle, two
trivially intersecting $2$-dimensional subbundles and a fibrewise map from the
pair
\[
  \left(
    \Ball{2}\times \Ball{2}
    ,
    \Sphere{1} \times \Ball{2}
    \cup
    \Ball{2} \times \Sphere{1}
    \setminus
    \left(
      \Sphere{1} \times \{0\}
      \cup
      \{0\} \times \Sphere{1}
    \right)
  \right)
\]
to $(\Group,H)$. This classifying space can be modeled as the homotopy quotient
of the mapping space
$
  \MappingSpace
    {
      \Ball{2}\times \Ball{2}
      ,
      \Sphere{1} \times \Ball{2}
      \cup
      \Ball{2} \times \Sphere{1}
      \setminus
      \left(
      \Sphere{1} \times \Ball{2}_\epsilon
      \cup
      \Ball{2}_\epsilon \times \Sphere{1}
      \right)
    }
    {
      \Group,H
    }
$
by its natural $\apply{\mathrm{SO}}{2}\times \apply{\mathrm{SO}}{2}$-action via
precomposition, where $\Ball{2}_\epsilon$ denotes some sufficiently small ball.
\begin{remark}
Note that in this case the $4$-dimensional bundle is
redundant; Nevertheless if one later wants to incorporate higher
self-intersections it becomes necessary.
\end{remark}
Now this classifying space admits a tautological $4$-dimensional bundle
as well, which we denote by $\TautologicalBundle{4,2,\Group,H}$, which contains
two $2$-dimensional subbundles $\TautologicalBundle{4,2,\Group,H}^1$ and
$\TautologicalBundle{4,2,\Group,H}^2$. It again carries a natural evaluation
which we also denote by $\Evaluation$.

Finally, it admits the following Thom like space:
\[
  \apply{T}{\TautologicalBundle{4,2,\Group,H}}
  \coloneqq
  \DiskBundle{\TautologicalBundle{4,2,\Group,H}}/
  \left(
    \Sphere{\TautologicalBundle{4,2,\Group,H}}
    \setminus
    \left(
      \SphereBundle{\TautologicalBundle{4,2,\Group,H}^1}_\epsilon
      \cup
      \SphereBundle{\TautologicalBundle{4,2,\Group,H}^2}_\epsilon
    \right)
  \right)
\]
Here $\SphereBundle{\TautologicalBundle{4,2,\Group,H}^i}_\epsilon$ denotes a
small neighborhood of $\SphereBundle{\TautologicalBundle{4,2,\Group,H}}^i$ in
$\SphereBundle{\TautologicalBundle{4,2,\Group,H}}$.

And again, by construction, $\Evaluation$ descends to a map
\[
  \overline{\Evaluation}
  \colon
  \apply{T}{\TautologicalBundle{4,2,\Group,H}}
  \to
  \MappingCone{\ClassifyingSpace{\Group}}{\ClassifyingSpace{H}}
\]
Note that $\SphereBundle{\TautologicalBundle{4,2,\Group,H}^1}_\epsilon\cup
\SphereBundle{\TautologicalBundle{4,2,\Group,H}^2}_\epsilon$ is a regular
neighborhood of a codimension two submanifold of
$\SphereBundle{\TautologicalBundle{4,2,\Group,H}}$ together with a map of pairs
\[
  \left(
    \SphereBundle{\TautologicalBundle{4,2,\Group,H}}
    ,
    \SphereBundle{\TautologicalBundle{4,2,\Group,H}}
    \setminus
    \left(
      \SphereBundle{\TautologicalBundle{4,2,\Group,H}^1}_\epsilon\cup
      \SphereBundle{\TautologicalBundle{4,2,\Group,H}^2}_\epsilon
    \right)
  \right)
  \to
  \left(\ClassifyingSpace{\Group},\ClassifyingSpace{H}\right)
\]

We denote the quotient
$
    \SphereBundle{\TautologicalBundle{4,2,\Group,H}}
    /
    \left(
      \SphereBundle{\TautologicalBundle{4,2,\Group,H}}
      \setminus
      \left(
        \SphereBundle{\TautologicalBundle{4,2,\Group,H}^1}_\epsilon\cup
        \SphereBundle{\TautologicalBundle{4,2,\Group,H}^2}_\epsilon
      \right)
    \right)
$
by
$\Boundary{\apply{T}{\TautologicalBundle{4,2,\Group,H}}}$.
This boundary space is a wedge of the two Thom spaces of the normal bundle of
$\SphereBundle{\TautologicalBundle{4,2,\Group,H}^1}$ and
$\SphereBundle{\TautologicalBundle{4,2,\Group,H}^2}$ in $
\SphereBundle{\TautologicalBundle{4,2,\Group,H}}$ and therefore admits a
map
\[
  \partial_{4,2}
  \colon
  \Boundary{\apply{T}{\TautologicalBundle{4,2,\Group,H}}}
  \to
  \ThomSpace{\TautologicalBundle{2,\Group,H}}.
\]

We denote the gluing
$
  \ThomSpace{\TautologicalBundle{2,\Group,H}}
  \cup_{\partial_{4,2}}
  \apply{T}{\TautologicalBundle{4,2,\Group,H}}
$
by
$X_{4,2,\Group,H}$. By construction both evaluations are compatible and
therefore extend to a map
$
  \overline{\Evaluation}
  \colon
  X_{4,2,\Group,H}
  \to
  \MappingCone{\ClassifyingSpace{\Group}}{H}
$.

Now it is easy to see that the proof of
Proposition~\ref{prp:RepresentabilityComplement}
adapts almost verbatim to prove the following:
\input{Proposition/RepresentabilityDoublePoints.tex}

In order to represent self-intersections with more layers, one has to consider
multiple two-dimensional bundles inside a 4-dimensional bundle and then proceed
the same. Note that if one wants to allow for self-intersections with different
numbers of layers, one has to glue different spaces to
$\ThomSpace{\TautologicalBundle{2,\Group,H}}$ (e.g. one for double points, one
for triple points, one for quadruple points etc.).

Finally Proposition~\ref{prp:TubNeighborhood} also allows a different kind of
singularity
which can also be modeled in the same way. In that case one has to consider
half-rays in a three dimensional space and built the associated Thom-like space
from their sphere bundles.

We want to emphasize here that it feels worthwhile to start making computations
for various groups, in particular free groups or their cousins like surface
groups and free products thereof, for these Thom spaces with
singularities. In this case the Thom isomorphism can be replaced by a simple
Serre spectral sequence computation and the associated Mayer-Vietoris sequence
is also manageable. The computations themselves are not really complicated but
rather tedious and ultimately this might provide a route to prove the
non-vanishing of the bounded cohomology of free groups.

%% file: Lemma/SubmanifoldComplementSubcomplex.tex
Let $\Manifold$ denote an $n$-dimensional manifold together with a codimension
$2$ submanifold $\Submanifold$. Note that by transversality, the inclusion
$\iota\colon \Manifold \setminus \Submanifold \to \Manifold$ induces a
surjection on fundamental groups.
\begin{lemma}
  \label{lem:Factorizes}
  In the above setting let, additionally, $\Group$ denote a discrete group and
  $\ContinuousMap \colon \Manifold \to \ClassifyingSpace{\Group}$ denote some
  map that induces a
  surjection $\HomotopyGroupOfObject{1}{\Manifold}{\ast}\to \Group$. This
  yields a map
  $
    \overline{\iota}
    \colon
    \ClassifyingSpace
      {\HomotopyGroupOfObject{1}{\Manifold\setminus\Submanifold}{\ast}}
    \to
    \ClassifyingSpace{\Group}
  $%
  .
  Then $\ContinuousMap$ is homotopic to a map of pairs
  \[
    \overline{\ContinuousMap}
    \colon
    \left(
      \Manifold,
      \Manifold
      \setminus
      \Submanifold
    \right)
    \to
    \left(\
      \ClassifyingSpace{\Group},
      \ClassifyingSpace
        {\HomotopyGroupOfObject{1}{\Manifold\setminus\Submanifold}{\ast}}
    \right)
  \]
\end{lemma}
\begin{proof}
  Since $\ClassifyingSpace{\Group}$ is
  aspherical the restriction of $\ContinuousMap$ to $\Manifold\setminus
  \Submanifold$ factorizes up to homotopy through
  $
    \ClassifyingSpace
      {\HomotopyGroupOfObject{1}{\Manifold\setminus\Submanifold}{\ast}}
  $
  Now the fact that the inclusion of $\Manifold \setminus
  \overline{\TubularNeighborhood{\Submanifold}}$ into $\Manifold$ is a
  cofibration yields the desired result.
  .
\end{proof}

%% file: Definition/MappingSpaceThomSpace.tex
\begin{definition}
  \label{def:ClassyfingSpaceWithMapping}
  Fix two groups $\Group$ and $H$ together with a surjection $s\colon H\to
  \Group$.
  Let
  $
    \ClassifyingSpaceWithMapping
      {\left(\ClassifyingSpace{\Group},\ClassifyingSpace{H}\right)}
  $
  denote the classifying space of oriented two-dimensional vector bundles
  that carry a map
  $
    (\Ball{2},\Sphere{1})
    \to
    (\ClassifyingSpace{\Group},\ClassifyingSpace{H})
  $
  as tangential data. A topological model of this classifying space is given by
  the homotopy quotient of
  $
    \MappingSpace
      {(\Ball{2},\Sphere{1})}
      {(\ClassifyingSpace{\Group},\ClassifyingSpace{H})}
  $
  by its natural $\apply{\mathrm{SO}}{2}$-action via precomposition.

  Via pullback this space carries a $2$-dimensional tautological bundle, which
  we denote by $\TautologicalBundle{2,\Group,H}$. Note that
  the evaluation
  \[
    \Evaluation
    \colon
    \MappingSpace
      {(\Ball{2},\Sphere{1})}
      {(\ClassifyingSpace{\Group},\ClassifyingSpace{H}}
    \times
    (\Ball{2},\Sphere{1})
    \to
    (\ClassifyingSpace{\Group},\ClassifyingSpace{H})
  \]
  yields a map
  \[
    \left(
      \DiskBundle{\TautologicalBundle{2,\Group,H}},
      \SphereBundle{\TautologicalBundle{2,\Group,H}}
    \right)
    \to
    \left(
      \ClassifyingSpace{\Group},
      \ClassifyingSpace{H}
    \right)
  \]
  which we also denote by $\Evaluation$.

  Finally this relative mapping, descends to a map from the Thom space
  $\ThomSpace{\TautologicalBundle{2,\Group,H}}$ to the mapping cone
  $\MappingCone{\ClassifyingSpace{\Group}}{H}$, which we denote by
  $\overline{\Evaluation}$.
\end{definition}

%% file: Proposition/RepresentabilityComplement.tex
\begin{proposition}
\label{prp:RepresentabilityComplement}
  Let $\Group$ denote a group and $H \to \Group$ a fixed surjection and let
  $\HomologyClass$ denote a rational homology class of $\Group$. Then the
  following are equivalent:
  \begin{itemize}
  \item
    There exists a manifold $\Manifold$, a map $\ContinuousMap\colon\Manifold
    \to \ClassifyingSpace{\Group}$ and a
    submanifold $\Submanifold\subset \Manifold$ such that $\ContinuousMap$
    maps the fundamental class of $\Manifold$ to a rational non-zero multiple
    of $\HomologyClass$ and the
    fundamental group of $\Manifold \setminus \Submanifold$ is isomorphic to
    $H$ and the restriction of $\ContinuousMap$ induces the fixed surjection.
  \item
    The image of $\HomologyClass$ in the rational homology of
    $\MappingCone{\ClassifyingSpace{\Group}}{\ClassifyingSpace{H}}$ is hit by
    \[
      \overline{\Evaluation}
      \colon
      \ThomSpace{\TautologicalBundle{2,\Group,H}}
      \to
      \MappingCone{\ClassifyingSpace{\Group}}{H}.
    \]
  \end{itemize}
\end{proposition}
\begin{proof}
  We have already established that if such a manifold exists, then the
  composition $\Manifold \to \MappingCone{\ClassifyingSpace{\Group}}{H}$
  factorizes through $\ThomSpace{\TautologicalBundle{2,\Group,H}}$, which
  establishes the first implication.

  Now assume that $\HomologyClass$ lies in the image of the homology of
  $\ThomSpace{\TautologicalBundle{2,\Group,H}}$. Pick a framed representative
  $f\colon \Manifold\to \ThomSpace{\TautologicalBundle{2,\Group,H}}$ of the
  preimage of
  $\HomologyClass$ which is transversal to the zero section
  in $\ThomSpace{\TautologicalBundle{2,\Group,H}}$. Let $\Submanifold$ denote
  the preimage of the zero section. Since the boundary of the tubular
  neighborhood $\Boundary{\TubularNeighborhood{\Submanifold}}$ is a two-sided
  codimension one submanifold of $\Manifold$ it
  is framed as well. Now additionally the long exact sequence of the pair
  $(\ClassifyingSpace{\Group},\ClassifyingSpace{H})$ shows that the
  composition of $f$ and $\Evaluation$ maps the fundamental class of
  $\Boundary{\TubularNeighborhood{\Submanifold}}$ to zero in the
  rational homology of $\ClassifyingSpace{H}$. Since framed bordism and
  singular homology agree rationally, there exists a multiple
  $k\cdot\Manifold$
  such that the map from $k\cdot\Boundary{\TubularNeighborhood{\Submanifold}}$
  to $\ClassifyingSpace{H}$ is framed nullbordant. Additionally, via surgery
  one can always arrange such a nullbordism to have fundamental group
  isomorphic to $H$.
  Let us denote such a
  nullbordism by $W$ and the corresponding map by $g_W$. Let us denote
  $\TubularNeighborhood{\Submanifold} \cup W$ by $\Manifold'$. By construction
  there exists a map $f'$ from $\Manifold'$ to $\ClassifyingSpace{\Group}$
  which
  maps the tubular neighborhood according to $\Evaluation$ and the nullbordism
  according to $g_W$ followed by the map inducing the fixed surjection. It
  follows from excision that $f'$ still represents the same rational homology
  class in the rational homology of
  $\MappingCone{\ClassifyingSpace{\Group}}{\ClassifyingSpace{H}}$. Therefore
  the image of the fundamental class of $\Manifold'$ in the rational homology
  of $\ClassifyingSpace{\Group}$ differs from $\HomologyClass$ by a class in
  the image of $\ClassifyingSpace{H}$. Evidently, every class stemming from
  $\ClassifyingSpace{H}$ is representable by a manifold as claimed by taking
  the submanifold to be empty. Therefore the other implication holds as well.
\end{proof}

%% file: Theorem/NoHyperbolicComplement.tex
\begin{proof}[Proof of Theorem~\ref{thm:CentralizerDimension}]
  We will prove the theorem by computing the homology of
  $\ThomSpace{\TautologicalBundle{2,\Group,H}}$. The Thom isomorphism
  identifies the homology of $\ThomSpace{\TautologicalBundle{2,\Group,H}}$ with
  the homology of
  $
    \ClassifyingSpaceWithMapping
      {\left(\ClassifyingSpace{\Group},\ClassifyingSpace{H}\right)}
  $%
  .
  As we have mentioned before a model of
  $
    \ClassifyingSpaceWithMapping
      {\left(\ClassifyingSpace{\Group},\ClassifyingSpace{H}\right)}
  $
  is given by the homotopy quotient of the mapping space
  $
    \MappingSpace
      {\left(\Ball{2},\Sphere{1}\right)}
      {\left(\ClassifyingSpace{\Group},\ClassifyingSpace{H}\right)}
  $
  by its natural $\apply{\mathrm{SO}}{2}$-action.
  Since $\ClassifyingSpace{\Group}$ is aspherical, it is easy to see that the
  restriction map
  $
    \MappingSpace
      {\left(\Ball{2},\Sphere{1}\right)}
      {\left(\ClassifyingSpace{\Group},\ClassifyingSpace{H}\right)}
    \to
    \MappingSpace{\Sphere{1}}{\ClassifyingSpace{H}}
  $
  is a $\apply{\mathrm{SO}}{2}$-equivariant homotopy equivalence on the
  connected components that it hits. Now a standard result in algebraic
  topology says that the connected components of
  $\MappingSpace{\Sphere{1}}{\ClassifyingSpace{H}}$ are given by the conjugacy
  classes of elements in $H$ and the connected component associated to a
  conjugacy class $[h]$ is the Eilenberg-MacLane space
  $\EMSpace{\apply{C_H}{h}}{1}$, where
  $\apply{C_H}{h}$ denotes the centralizer of $h$ in $H$.

  Now we have the following exact sequence
  \[
    1 \to \langle h\rangle \to \apply{C_H}{h} \to \apply{C_H}{h}/\langle
    h\rangle
    \to 1.
  \]
  Since $H$ is torsion-free $\langle h\rangle$ is infinite cyclic,
  additionally it is easy
  to see that the $\apply{\mathrm{SO}}{2}$-action on the centralizer
  corresponds to rotation along the $\Sphere{1}$-fiber of the classifying space
  of the centralizer. In particular this action is free. Therefore the homotopy
  quotient is given by $\ClassifyingSpace{\apply{C_H}{h}/\langle h\rangle}$,
  which has
  dimension one less than the cohomological dimension of the centralizer
  $\apply{C_H}{h}$.

  Therefore
  $
    \ClassifyingSpaceWithMapping
      {\left(\ClassifyingSpace{\Group},\ClassifyingSpace{H}\right)}
  $
  is homotopy equivalent to a disjoint union of classifying spaces of
  $\apply{C_H}{h}/\langle h\rangle$ for all the non-zero conjugacy classes of
  elements in
  the kernel of $s\colon H \to \Group$ and one copy of $\ComplexNumbers
  P^\infty$ for
  the trivial conjugacy class. Therefore the Thom space is homotopy equivalent
  to a wedge of Thom spaces of the tautological bundles on the classifiyng
  spaces of the centralizers and the Thom
  space of the tautological bundle of
  $\ComplexNumbers P^\infty$ (which is homotopy equivalent to $\ComplexNumbers
  P^\infty$ as well). Since the projective space component corresponds to the
  trivial conjugacy class, the corresponding null homotopy in
  $\ClassifyingSpace{\Group}$ is constant as well, meaning that this whole
  component maps to the cone point in
  $\MappingCone{\ClassifyingSpace{\Group}}{H}$. In particular its homology is
  mapped to zero in the mapping cone. Therefore we conclude that the image of
  the homology of the Thom space is non-zero in degrees at most $(c-1) + 2$,
  where
  $c$ denotes the maximal rational cohomological dimension of an element in the
  centralizer of the kernel of $s$.
\end{proof}
Note that the computation carried out in this proof is a general blue print for
computing these Thom spaces.

%% file: Proposition/RepresentabilityDoublePoints.tex
\begin{proposition}
  \label{prp:RepresentabilityComplementDoublePoints}
  Let $\Group$ denote a group and $H \to \Group$ a fixed surjection and let
  $\HomologyClass$ denote a rational homology class of $\Group$. Then the
  following are equivalent:
  \begin{itemize}
    \item
    There exists a manifold $\Manifold$, a map $\ContinuousMap\colon\Manifold
    \to \ClassifyingSpace{\Group}$ and an immersed
    submanifold $\Submanifold\subset \Manifold$ which only has transverse
    double points as self-intersections, such that
    $\ContinuousMap$
    maps the fundamental class of $\Manifold$ to a rational non-zero multiple
    of $\HomologyClass$ and the
    fundamental group of $\Manifold \setminus \Submanifold$ is isomorphic to
    $H$ and the restriction of $\ContinuousMap$ induces the fixed surjection.
    \item
    The image of $\HomologyClass$ in the rational homology of
    $\MappingCone{\ClassifyingSpace{\Group}}{\ClassifyingSpace{H}}$ is hit by
    \[
      \overline{\Evaluation}
      \colon
      X_{4,2,\Group,H}
      \to
      \MappingCone{\ClassifyingSpace{\Group}}{H}.
    \]
  \end{itemize}
\end{proposition}

%% file: sources.bib
@article{GromovBoundedCohomology,
     author = {Gromov, Michael},
     title = {Volume and bounded cohomology},
     journal = {Publications Math\'ematiques de l'IH\'ES},
     pages = {5--99},
     publisher = {Institut des Hautes \'Etudes Scientifiques},
     volume = {56},
     year = {1982},
     mrnumber = {84h:53053},
     zbl = {0516.53046},
     language = {en},
     url = {http://www.numdam.org/item/PMIHES_1982__56__5_0/}
}

@article {MatsumotoMoritaBoundedCohomology,
    AUTHOR = {Matsumoto, Shigenori and Morita, Shigeyuki},
     TITLE = {Bounded cohomology of certain groups of homeomorphisms},
   JOURNAL = {Proc. Amer. Math. Soc.},
  FJOURNAL = {Proceedings of the American Mathematical Society},
    VOLUME = {94},
      YEAR = {1985},
    NUMBER = {3},
     PAGES = {539--544},
      ISSN = {0002-9939},
   MRCLASS = {55N99 (57T99)},
  MRNUMBER = {787909},
       DOI = {10.2307/2045250},
       URL = {https://doi.org/10.2307/2045250},
}

@article{SomaThird,
title = {Bounded cohomology of closed surfaces},
journal = {Topology},
volume = {36},
number = {6},
pages = {1221-1246},
year = {1997},
issn = {0040-9383},
doi = {https://doi.org/10.1016/S0040-9383(97)00003-7},
url = {https://www.sciencedirect.com/science/article/pii/S0040938397000037},
author = {Teruhiko Soma}
}

@inbook{BrooksSecondBC,
url = {https://doi.org/10.1515/9781400881550-006},
title = {Some Remarks on Bounded Cohomology},
booktitle = {Riemann Surfaces and Related Topics (AM-97), Volume 97},
booktitle = {Proceedings of the 1978 Stony Brook Conference. (AM-97)},
author = {Robert Brooks},
editor = {Irwin Kra and Bernard Maskit},
publisher = {Princeton University Press},
address = {Princeton},
pages = {53--64},
doi = {doi:10.1515/9781400881550-006},
isbn = {9781400881550},
year = {1981},
lastchecked = {2025-03-06}
}

@BOOK{FrigerioBook,
  title     = "Bounded cohomology of discrete groups",
  author    = "Frigerio, Roberto",
  publisher = "American Mathematical Society",
  series    = "Mathematical Surveys and Monographs",
  month     =  dec,
  year      =  2017,
  address   = "Providence, RI",
  language  = "en"
}

@ARTICLE{WalshGreatCircleLinks,
  title     = "Great circle links and virtually fibered knots",
  author    = "Walsh, Genevieve S",
  journal   = "Topology",
  publisher = "Elsevier BV",
  volume    =  44,
  number    =  5,
  pages     = "947--958",
  month     =  sep,
  year      =  2005,
  copyright = "https://www.elsevier.com/open-access/userlicense/1.0/",
  language  = "en"
}

@BOOK{FrigerioMoraschiniMutliComplexes,
  title     = "Gromov's theory of multicomplexes with applications to bounded
               cohomology and simplicial volume",
  author    = "Frigerio, Roberto and Moraschini, Marco",
  publisher = "American Mathematical Society",
  series    = "Memoirs of the American Mathematical Society",
  month     =  may,
  year      =  2023,
  address   = "Providence, RI",
  language  = "en"
}

@ARTICLE{KastenholzOpenBooks,
  title         = "Simplicial volume of open books in dimension 4",
  author        = "Kastenholz, Thorben",
  month         =  apr,
  year          =  2025,
  copyright     = "http://creativecommons.org/licenses/by-nc-sa/4.0/",
  archivePrefix = "arXiv",
  primaryClass  = "math.GT",
  eprint        = "2504.10975"
}

@ARTICLE{3ManifoldGroups,
  title         = "Two aspects of graph 3-manifold groups",
  author        = "Sun, Hongbin",
  month         =  jul,
  year          =  2026,
  copyright     = "http://creativecommons.org/licenses/by/4.0/",
  archivePrefix = "arXiv",
  primaryClass  = "math.GT",
  eprint        = "2607.04326"
}

@book{MilnorStasheff,
  title     = {Characteristic Classes},
  author    = {Milnor, John W. and Stasheff, James D.},
  series    = {Annals of Mathematics Studies},
  volume    = {76},
  year      = {1974},
  publisher = {Princeton University Press},
  address   = {Princeton, N.J.},
  isbn      = {978-0-691-08122-9},
  doi       = {10.1515/9781400881826}
}
